\documentclass[a4paper,12pt]{article}
\usepackage{amsmath}
\usepackage{amssymb}
\usepackage{setspace}
\usepackage{fullpage}
\usepackage{centernot}
\usepackage{comment}

\usepackage{bbm}
\usepackage{bm}
\usepackage{amsthm}
\usepackage{pdfsync}
\usepackage{hyperref}
\newtheorem{theorem}{Theorem}[section]
\newtheorem{lemma}[theorem]{Lemma}

\newtheorem{corollary}[theorem]{Corollary}
\newtheorem{result}[theorem]{Result}
\newtheorem{proposition}[theorem]{Proposition}
\theoremstyle{definition}
\usepackage[normalem]{ulem}
\usepackage{xcolor}

\newtheorem{definition}{Definition}

\newtheorem{remark}[theorem]{Remark}

\newcommand{\stkout}[1]{\ifmmode\text{\sout{\ensuremath{#1}}}\else\sout{#1}\fi}
\def\PG{\mathrm{PG}} 
\def\F{\mathbb{F}}

\def\Aut{\mathrm{Aut}}
\def\PGammaL{\mathrm{P}\Gamma\mathrm{L}}

\def\PGL{\mathrm{PGL}}

\def\FF{\mathbb{F}}

\def\Fq{\mathbb{F}_q}
\def\Fn{\mathbb{F}_{q^n}}

\def\Tr{\mathrm{Tr}}
\def\Gal{\mathrm{Gal}}

\def\cha{\mathrm{char}}
\def\K{\mathbb{K}}
\def\L{\mathbb{L}}
\def\P{\mathbb{P}}
\def\QQ{\mathbb{Q}}

\def\EQ1{[EQUIVALENCE 1]}
\def\EQ2{[EQUIVALENCE 2]}

\def\la{\langle}
\def\ra{\rangle}

\def\S{\mathcal{S}}

\def\E{\mathbb{E}}

\title{On the product of elements with prescribed trace}
\author{John Sheekey \and Geertrui Van de Voorde \and Jos\'e Felipe Voloch}
\begin{document}
\maketitle
\begin{abstract} This paper deals with the following problem. Given a finite extension of fields $\L/\K$ and
denoting the trace map from $\L$ to $\K$ by $\Tr$, for which elements $z$ in 
$\L$, and $a$, $b$ in $\K$, is it possible to write $z$ as a product $x\cdot y$, where $x,y\in \L$ with $\Tr(x)=a, \Tr(y)=b$?
We solve most of these problems for finite fields, with a complete solution when the degree of the extension is at least $5$. We also have results for arbitrary fields and extensions of degrees $2,3$ or $4$. 
We then apply our results to the study of PN functions, semifields, irreducible polynomials with prescribed coefficients,
and to a problem from finite geometry concerning the existence of certain disjoint linear sets.
\end{abstract}


\section{Introduction}

The goal of this paper is to find out under which conditions it is possible to write an element of a field as the product of two elements with prescribed trace with respect to a subfield. Our main theorem (Theorem \ref{Th1}) states that, for finite fields $\F_{q^n}$, $n\geq 5$, every element of $\F_{q^n}^*$ can be written as the product of two elements of prescribed trace. In order to obtain this result, we translate the problem into the problem of finding rational points on an associated curve.

Our question was motivated by a number of applications which are gathered in Section \ref{applications}. A first application deals with the existence problem for particular disjoint {\em linear sets} on a projective line. These results are very relevant in the study of many problems from finite geometry, in particular for (multiple) blocking sets \cite{blo2,blo} and KM-arcs \cite{KM}. The second application of our main theorem leads to an improvement on a result of Kyureghyan and Ozbudak \cite{KyOz} about {\em PN functions}. An adaptation of the proof of our main theorem leads to a result on the non-existence of certain {\em semifields of BEL rank $2$}. Finally, we link our problem to the study of {\em polynomials with certain prescribed coefficients}.

Our applications use only the case of finite fields, so we concentrate on this case. But, when it is no further effort to do so, we work in greater generality. 

\subsection{The trace map}
Let $\L$ be a field with subfield $\K$ and as usual, denote the degree of $\L$ over $\K$ by $[\L:\K]$.
Let $\Tr_{\L/\K}$ denote the field trace from $\L$ to $\K$, i.e. $\Tr_{\L/\K}(\alpha)$ is the trace of the $\K$-linear map on $\L$, $x\mapsto \alpha x$. If the fields $\L$ and $\K$ are clear from the context, we may write $\Tr$ instead, and in the case that $\K$ is the finite field of order $q$, $\F_q$ and $\L=\F_{q^n}$, $\Tr_{\F_{q^n}/\F_q}$ is simply denoted by $\Tr_n$.
We remark that the trace is
identically zero for purely inseparable extensions and our problem is trivial in this case. 

The fact that $\Tr_{\L/\K}$ is an $\K$-linear map leads to the following observation which is easy to prove.
\begin{lemma} \label{trivial} Let $a\in \K^*$ and $b\in \K^*$.
It is possible to write $z\in \L^*$ as a product $x\cdot y$, where $x,y\in \L$ with $\Tr(x)=a$ and $\Tr(y)=b$ if and only if it is possible to write $\frac{z}{ab}$ as a product $x'\cdot y'$ with $\Tr(x')=1$ and $\Tr(y')=1$. 

It is possible to write $z\in \L$ as a product $x\cdot y$, where $x,y\in \L$ with $\Tr(x)=0$ and $\Tr(y)=b$, $b\in \K^*$ if and only if it is possible write $z$ as a product $x' \cdot y'$ with $\Tr(x')=0$ and $\Tr(y')=1$.

\end{lemma}

We define the subsets of elements with prescribed trace, and the sets of products of elements with prescribed trace as follows.
\begin{definition} $T_a=\{x \in \L\mid \Tr(x)=a\}$. $T_aT_b=\{x\cdot y\mid x\in T_a,y \in T_b\}$.
\end{definition}

The following corollary can easily be deduced from Lemma \ref{trivial}.
\begin{corollary}\label{Ts} Let $a,b\in \K^*$. We have $T_0T_1=T_0T_a$, $T_aT_b=T_1T_{ab}$ and $z\in T_aT_b$ if only if $\frac{z}{ab}\in T_1T_1$. 
\end{corollary}

This tells us that, in order to investigate whether all elements can be written as the product of two elements with prescribed trace, we may restrict ourselves to investigating the sets $T_0T_0$, $T_0T_1$ and $T_1T_1$.

The trace of an element can be calculated as follows. Let $\K(a)$ denote the simple field extension of $\K$ generated by the element $a$.

\begin{result}\label{def} For $a$ in $\L$, let $\sigma_1(a),\ldots,\sigma_n(a)$ be be the roots (counted with multiplicity) of the minimal polynomial of $a$ over $\K$ (in some extension field of $\L$), then

\[ \Tr _{\L/\K}(a)=[\L:\K(a)]\sum _{j=1}^{n}\sigma _{j}(a ).\]
\end{result}

We also recall the following link between the trace and norm of an element and the coefficients of its minimal polynomial:

\begin{result} Let the minimal polynomial for $\alpha$ over $\K$ be $X^d+c_{d-1}X^{d-1}+\cdots+c_1X+c_0$.
Then $\Tr_{\K(\alpha)/\K}(\alpha)=-c_{d-1}$ and $\mathrm{N}_{\K(\alpha)/\K}(\alpha)=(-1)^dc_0$.

\end{result}

Finally we recall the following well known property of the trace function for separable extensions.
\begin{result} \label{nondegen}
Suppose $\L$ is a separable extension of $\K$. Then the function $(x,y)\mapsto \Tr _{\L/\K}(xy)$ is a nondegenerate symmetric $\K$-bilinear form from $\L\times \L$ to $\K$.
\end{result}

\subsection{Overview}

The first part of our paper deals with field extensions of degree $2$ (Section \ref{sec2}), degree $3$ (Section \ref{sec3}) and degree $4$ (Section \ref{sec4}). The main theorem \ref{Th1} appears in Section \ref{sec5}. In Sections \ref{sec2}, \ref{sec3}, \ref{sec4}, the proofs are split in several seperate cases. We give an overview here.
\begin{itemize}
\item Extensions of degree $2$: 
\begin{itemize}
\item  $T_0T_0\neq \L$: Proposition \ref{n2T0T0} for $\cha(\K)\neq 2$, Proposition \ref{n2T0T02} for $\cha(\K)= 2$
\item $T_0T_1\neq \L$: Proposition \ref{n2T0T1} for $\cha(\K)\neq 2$, Proposition \ref{n2T0T12} for $\cha(\K)= 2$
\item  $T_1T_1\neq \L$: Proposition \ref{n2T1T1} for $\cha(\K)\neq 2$, Proposition \ref{n2T1T12} for $\cha(\K)= 2$, Corollary \ref{n2finite} for finite fields with $\cha(\K)= 2$
\item When is $1\in T_0T_0,T_0T_1,T_1T_1?$: Corollary \ref{n2-one}.
\end{itemize}

\item Extensions of degree $3$
\begin{itemize}
\item  $T_0T_0\neq \L$: General condition: Proposition \ref{prop:cyclicgalois} and Proposition \ref{prop:cubicT0T0}. Corollary \ref{gevolg1} for $\cha(\K)\neq 2,3$, Corollary \ref{n3T0T0finite2} for finite fields with $\cha(\K)=2$, Corollary \ref{n3T0T0finite3} for finite fields with $\cha(\K)=3$
\item When is $1\in T_0T_0$?: Corollary \ref{n3one} for arbitrary fields and Corollary \ref{gevolg3} for finite fields
\item $T_0T_1= \L$: Proposition \ref{prop:T0T1deg3}  for $\K(\alpha)$, $\alpha$ root of $X^3-d=0$
\item $T_1T_1$: General condition: Lemma \ref{n3T1T1}

\item When is $1\in T_1T_1?$: Proposition \ref{T1T1one}.
\end{itemize}

\item Extensions of degree $4$
\begin{itemize}
\item $T_0T_0=\L$: For $\L=\K(\alpha)$, $\alpha^4+\alpha^2+d=0$: Proposition \ref{T0T04}.  For $\L=\K(\alpha)$, $\alpha^4-d=0$: Proposition \ref{n4T0T0}. Corollary \ref{n4T0T0fin2} for $\F_{q^n}$, $q\equiv 1,3\mod 4$
\item When is $1\in T_0T_0?$: Proposition \ref{n4fieldeven} for $\cha(\K)=2$, Proposition \ref{n4subfield} for $\cha(\K)\neq 2$, Corollary \ref{n4finone} for finite fields
\item $T_0T_1=\L$: For $\L=\K(\alpha)$, $\alpha^4+\alpha^2+d=0$: Proposition \ref{T0T14a}. For $\L=\K(\alpha)$, $\alpha^4-d=0$: Proposition \ref{T0T14}. Corollary \ref{n4T0T1fin2} for $\F_{q^4}$, $q\equiv 1,3\mod 4$
\item For finite fields $\F_{2^4}$ and $\F_{3^4}$: $T_aT_b=\L^*\cup\{ab\}$ see Theorem \ref{Th1}.
\end{itemize}

\item Extensions of degree at least $5$ for finite fields: $T_aT_b=\L^*\cup\{ab\}$ see Theorem \ref{Th1}.

\end{itemize}

\begin{remark} We were not able to derive a result for $T_1T_1$ for field extensions of degree $4$. In particular, we do not know whether $\F_{q^4}^*=T_1T_1$ when $q>3$.

\end{remark}

\section{Extensions of degree $2$}\label{sec2}

\begin{lemma} \label{lemma2} Let $\K$ be an arbitrary field and let $f:X^2+a_1X+a_2$ be an irreducible polynomial over $\K$. Let $\alpha$ be a root of $f$, let $\L=\K(\alpha)$ and let $x=x_0+x_1\alpha$ be an arbitrary element in $\L$. Then $$\Tr_{\L/\K}(x)=2x_0-a_1x_1.$$
\end{lemma}
\begin{proof} 
Since $\Tr$ is $\K$-linear, $\Tr(x) = x_0\Tr(1)+ x_1\Tr(\alpha)$ and the result follows as $\Tr(1)=2, \Tr(\alpha)=-a_1$.

\end{proof}

\subsection{Extensions of degree $2$ for $\cha(\K)\neq 2$}
\subsubsection{$T_0T_0$}

\begin{proposition}\label{n2T0T0} Let $\L$ and $\K$ be as in Lemma \ref{lemma2} with $\mathrm{char}(\K)\neq 2$,  and let $\beta$ in $\L$. Then $\beta\in T_0T_0$ if and only if $\beta\in \K$.
\end{proposition}
\begin{proof} Let $X^2+a_1X+a_2$ be the minimal polynomial of $\alpha$. Let $\beta=b_0+b_1\alpha$, where $b_0,b_1\in \K$. If $x \in \L$ has $\Tr(x)=0$, then, by Lemma \ref{lemma2}, $x=x_0+x_1\alpha$ with $2x_0-a_1x_1=0$, $x_0,x_1\in \K$. Similarly, $y\in L$ has $\Tr(y)=0$ if and only if $y=y_0+y_1\alpha$ with $2y_0-a_1y_1=0$, $y_0,y_1\in \K$. It folllows that
$x\cdot y = \beta$ is equivalent to
$$x_1y_1\left(\frac{a_1}{2}+\alpha\right)^2=b_0+b_1\alpha.$$
The left hand side of this expression equals 
$$x_1y_1\left(\frac{a_1^2}{4}-a_2\right)$$ which is an element of $\K$. 
Hence, if $\beta\in \L\setminus \K$, there are no elements $x,y\in \L$ with $\Tr(x)=\Tr(y)=0$ and $x\cdot y = \beta$. If $\beta=b_0\in \K$, then putting $x_1=1$ and $y_1=\frac{b_0}{a_1^2/4-a_2}$ yields the solution $x=\frac{a_1}{2}+\alpha$ with $\Tr(x)=0$ and $y=\frac{b_0}{a_1^2/4-a_2}(\frac{a_1}{2}+\alpha)$ with $\Tr(y)=0$ to $x\cdot y=\beta$. Note that the denominator $a_1^2/4-a_2$ is nonzero, as otherwise $X^2+a_1X+a_2$ would not be irreducible.
\end{proof}
\subsubsection{$T_0T_1$}

\begin{proposition} \label{n2T0T1}Let $\L$ and $\K$ be as in Lemma \ref{lemma2} with $\mathrm{char}(\K)\neq 2$ and let $\beta$ in $\L$. Then $\beta\in T_0T_1$ if and only if $\beta\in \L\setminus \K^*$.
\end{proposition}
\begin{proof} We proceed as in the proof of Proposition \ref{n2T0T0}. We have that $x$ is of the form $x_1(\frac{a_1}{2}+\alpha)$, $y$ is of the form $\frac{1}{2}+y_1(\frac{a_1}{2}+\alpha)$ for some $x_1,y_1\in \K$. Putting $x\cdot y=b_0+b_1\alpha$ yields
\begin{align*}
\frac{a_1x_1}{4}+\frac{x_1y_1a_1^2}{4}-a_2x_1y_1&=b_0\\
\frac{x_1}{2}&=b_1.
\end{align*}
If $\beta\in \K^*$, then $b_1=0$, forcing $x_1=0$ and $b_0=0$, a contradiction. If $\beta\in \L\setminus \K^*$, then $b_1\neq 0$. We see that $x=2b_1$ and $y=(\frac{b_0}{x_1}-\frac{a_1}{4})\frac{4}{a_1^2-4a_2}$ yields a solution to $x\cdot y=\beta$ with $\Tr(x)=0$ and $\Tr(y)=1$.
\end{proof}

\subsubsection{$T_1T_1$}
\begin{proposition}\label{n2T1T1} Let $\L$ and $\K$ be as in Lemma \ref{lemma2} with $\mathrm{char}(\K)\neq 2$ and let $\beta=b_0+b_1 \alpha$ in $\L$. Then $\beta\in T_1T_1$ if and only if $\frac{(a_1b_1+1)^2-4(b_0+a_2b_1^2)}{a_1^2-4a_2}$ is a square.
\end{proposition}
\begin{proof} 
We proceed as in the proof of Proposition \ref{n2T0T0}. We have that $x$ is of the form $\frac{1}{2}+x_1(\frac{a_1}{2}+\alpha)$ and $y$ is of the form $\frac{1}{2}+y_1(\frac{a_1}{2}+\alpha)$ for some $x_1,y_1\in \K$. Putting $x\cdot y=b_0+b_1\alpha$ yields
\begin{align*}
\frac{1}{4}+\frac{1}{2}\left(\frac{a_1}{2}+\alpha\right)(x_1+y_1)+x_1y_1\left(\frac{a_1^2}{4}-a_2\right)&=b_0\\
\frac{1}{2}(x_1+y_1)&=b_1.
\end{align*}
We see that $x_1,y_1$ are the solutions to the quadratic equation
$$Y^2-2b_1Y+\left(\frac{b_0}{4}-\frac{2a_1b_1}{4}\right)\frac{4}{a_1^2-4a_2}=0.$$
This implies that $x_1,y_1$ can be found if and only if $\frac{(a_1b_1+1)^2-4(b_0+a_2b_1^2)}{a_1^2-4a_2}$ is a square
\end{proof}

\subsection{Extensions of degree $2$, $\cha(\K)=2$}

First note that if $\cha(\K)=2$, an extension $\K(\alpha)/\K$ with $\alpha$ a root of $X^2+a_1X+a_2$ with $a_1=0$ is inseparable. 

Hence, we may restrict ourselves to $\K(\alpha)$, $\alpha$ a root of $X^2+a_1X+a_2$ with $a_1\neq 0$.
\subsubsection{$T_0T_0$}

\begin{proposition} \label{n2T0T02}Let $\K$ be an arbitrary field of characteristic $2$ and let $f:X^2+a_1X+a_2$ be an irreducible polynomial over $\K$ with $a_1\neq 0$. Let $\alpha$ be a root of $f$, let $\L=\K(\alpha)$ and let $\beta$ be an arbitrary element in $\L$.  Then $\beta\in T_0T_0$ if and only if $\beta \in \K$.
\end{proposition}
\begin{proof} Let $x=x_0+\alpha x_1$ be an element of $\L$ with $\Tr(x)=0$. Since $\Tr(x)=a_1x_1$ by Lemma \ref{lemma2}, and $a_1\neq 0$, $x_1=0$. Similarly, if $\Tr(y_0+\alpha y_1)=0$, we have that $y_1=0$. Hence, if $x\cdot y=\beta$, then $\beta \in \K$. If $\beta\in \K$, then putting $x=1$ and $y=\beta$ yields that $\beta\in T_0T_0$.
\end{proof}

\subsubsection{$T_0T_1$}
\begin{proposition} \label{n2T0T12}Let $\L$\ and $\K$ be as in Proposition \ref{n2T0T02} and let $\beta=\beta_0+\beta_1\alpha$ be an arbitrary element in $\L$. Then $\beta\in T_0T_1$ if and only if $\beta \in \L\setminus \K^*$.
\end{proposition}
\begin{proof} Let $x=x_0+\alpha x_1$ be an element of $\L$ with $\Tr(x)=0$. Since $\Tr(x)=a_1x_1$, and $a_1\neq 0$, $x_1=0$. If $\Tr(y_0+\alpha y_1)=1$, we have that $y_1=1/a_1$. Hence, $x\cdot y=\beta$, yields that 
\begin{align*}
x_0y_0=b_0\\
\frac{x_0}{a_1}=b_1.
\end{align*}
If $b_1=0$, then $x_0=0$ and $b_0=0$. Hence, if $x\cdot y=\beta$, then $\beta=0$ or $\beta\in \L\setminus \K$. If $\beta\in \L\setminus \K$, putting $x_0=a_1b_1$ and $y_0=\frac{b_0}{x_1}$ yields $x=a_1b_1\in T_0$ and $y=\frac{b_0}{x_0}+\frac{1}{a_1}\alpha\in T_1$ as solution to $x\cdot y=\beta$.
\end{proof}

\subsubsection{$T_1T_1$}

\begin{proposition} \label{n2T1T12}Let $\L$\ and $\K$ be as in Proposition \ref{n2T0T02} and let $\beta=\beta_0+\beta_1\alpha$ be an arbitrary element in $\L$. Then $\beta \in T_1T_1$ if and only if the quadratic equation
$$X^2+(a_1b_1+1)X+\frac{a_2}{a_1^2}+b_0=0$$
has a solution.
\end{proposition}
\begin{proof} Let $x=x_0+\alpha x_1$ be an element of $\L$ with $\Tr(x)=1$. Since $\Tr(x)=a_1x_1$, we have that $x_1=1/a_1$. Similarly, if $\Tr(y_0+\alpha y_1)=1$, we have that $y_1=1/a_1$. Hence, $x\cdot y=\beta$, yields that 
\begin{align*}
x_0y_0+\frac{a_2}{a_1^2}=b_0\\
\frac{x_0+y_0+1}{a_1}=b_1.
\end{align*}

This shows that $x_0,y_0$
 are the solutions of the quadratic equation $$X^2+(a_1b_1+1)X+\frac{a_2}{a_1^2}+b_0=0.$$
\end{proof}

In particular, when $\K$ is a finite field, we necessarily have that $a_1\neq 0$ and we obtain the following.
\begin{corollary}\label{n2finite}Let $\K$ be a finite field of characteristic $2$ and let $f:X^2+a_1X+a_2$ be an irreducible polynomial over $\K$. Let $\alpha$ be a root of $f$, let $\L=\K(\alpha)$ and let $\beta=\beta_0+\beta_1\alpha$ be an arbitrary element in $\L$. Then $\beta\in T_1T_1$ if and only if

$$\Tr_{\F_{q^2}/\F_q}\left(\frac{\frac{a_2}{a_1^2}+b_0}{(a_1b_1+1)^2}\right)=0.$$
\end{corollary}

\subsection{When is $1\in T_aT_b$?}
\begin{corollary} \label{n2-one}Let $[\L:\K]=2$. Then $1\in T_0T_0$ and $1\notin T_0T_1$. 

If $\cha(\K)\neq 2$ with $\L=\K(\alpha)$ where $\alpha$ is a root of $X^2+a_1X+a_2=0$, then $1\in T_1T_1$ if and only if $\frac{(a_1+1)^2-4}{a_1^2-4a_2}$ is a square.

If $\cha(\K)= 2$ with $\L=\K(\alpha)$ where $\alpha$ is a root of $X^2+a_1X+a_2=0$, $a_1\neq 0$, then $1\in T_1T_1$ if and only if $X^2+X+\frac{a_2}{a_1}+1=0$ has a solution in $\K$.

\end{corollary}
\section{Extensions of degree $3$}\label{sec3}

In this section we consider extensions of degree three over an arbitrary field $\K$. We first show some general results, and then give explicit results for specific cases.

We distinguish three categories of extensions of degree three.
\begin{itemize}
\item Galois extensions: $\L=\K(\alpha)$, with the minimal polynomial of $\alpha$ having three distinct roots in $\L$.
\item Separable extensions which are not Galois: $\L=\K(\alpha)$, with the minimal polynomial of $\alpha$ having one root in $\L$, and three distinct roots in $\L(\omega)$ for some $\omega\notin\L$.
\item Inseparable extensions: $\L=\K(\alpha)$, with the minimal polynomial of $\alpha$ over $\K$ having repeated roots in $\L$
\end{itemize}
The last of these can only occur in characteristic three, with $\alpha$ being a triple root of its minimal polynomial. In this case, the trace function is identically zero, and hence the questions become trivial.

\subsection{Separable extensions of degree $3$}
The following proposition will assist in our consideration of $T_0T_0$ and $T_0T_1$ for separable extensions.

\begin{lemma}\label{prop:Trquadform}
Suppose $\L$ is a separable extension of $\K$ of degree three with $\cha(\K)\neq 2$, and let $(x_0,x_1,x_2)$ be the coordinates of $x\in \L$ with respect to some $\K$-basis. Then the coordinates of the elements of the set $\{x\in \L^*:\Tr_{\L/\K}(\beta x^{-1})=0\}$ are precisely the zeroes of a nondegenerate quadratic form.
\end{lemma}

\begin{proof}
Let $\L=\K(\alpha)$, and let $f$ be the minimal polynomial of $\alpha$ over $\K$. Let $\E := \K(\alpha,\omega)$ be the splitting field of $f$. Then there is an element $\sigma$ generating a subgroup of $\Gal(\E/\K)$ of order $3$, with fixed field $\K(\omega)$, and $\E$ is a Galois extension of $\K(\omega)$. Note that this includes the case where $\E=\L$, i.e. $\L$ is Galois.

For any element $x$ of $\L$, its minimal polynomial over $\K$ is equal to its minimal polynomial over $\K(\omega)$. Hence we have that
\begin{align*}
\Tr_{\L/\K}(x) &= \Tr_{\E/\K(\omega)}(x) = x+x^\sigma+x^{\sigma^2},\\
N_{\L/\K}(x) &= N_{\E/\K(\omega)}(x) = x^{1+\sigma+\sigma^2}.
\end{align*}

Now suppose $\Tr_{\L/\K}(\beta x^{-1})=0$, with $x\ne 0$. Then multiplying both sides of this equation by $N_{\L/\K}(x)$, we get that
\[
C_\beta(x) := \Tr_{\E/\K(\omega)}(\beta x^{\sigma+\sigma^2})=0.
\]
It is straightforward to see that $C_\beta(x)$ is a quadratic form over $\K$. The bilinear form $B_\beta(x,y) = \Tr_{\E/\K(\omega)}(\beta x^\sigma y^{\sigma^2})=\Tr_{\L/\K}(\beta x^\sigma y^{\sigma^2})$ is nondegenerate by Result \ref{nondegen}, and $C_\beta(x) = B_\beta(x,x)$, implying that $C_\beta(x)$ is a nondegenerate quadratic form. 
\end{proof}

\begin{corollary}\label{eps}
Suppose $\L$ is a separable extension of $\K$ of degree three. Then $\beta \in T_0T_\epsilon$, $\epsilon\in \{0,1\}$ if and only if the quadratic form $C_\beta(x)=0$ and the linear form $\Tr_{\L/\K}(x)=\epsilon$ have a nonzero common solution. 
\end{corollary}

\subsection{$\K(\alpha)$ Galois, $\alpha$ a root of $x^3-\lambda x^2+\mu x - \nu=0$. }\label{lambdanietnul}
\subsubsection{$T_0T_0$} 
\begin{proposition}\label{prop:cyclicgalois}
Suppose $\L$ is a cubic Galois extension of $\K$, $\cha(\K)\neq 2$. Let $\alpha_0\in \L$ be such that the roots $\alpha_0,\alpha_0^\sigma,\alpha_0^{\sigma^2}$ of the minimal polynomial $f=x^3-\lambda x^2+\mu x - \nu$ of $\alpha_0$ form a $\K$-basis for $\L$. Given $\beta=b_0\alpha_0+b_1\alpha_0^\sigma+b_2\alpha_0^{\sigma^2}$, then $\beta\in T_0T_0$ if and only if

\[
(\lambda^2 - 2\mu)(b_0^2+b_1^2+b_2^2)-2(\lambda^2 - 4\mu)(b_0b_1+b_0b_2+b_1b_2)
\]
is a square in $\K$.\end{proposition}

\begin{proof}
Let $\Gal(\L/\K) = \langle\sigma\rangle$, and let $\alpha_1=\alpha_0^\sigma,\alpha_2=\alpha_0^{\sigma^2}$ be the other roots of $f$. Note that there always exists an $\alpha_0\in \L$ such that $\{\alpha_0,\alpha_1,\alpha_2\}$ is a $\K$-basis for $\L$; this follows from the Normal Basis Theorem. 

Let $x=x_0\alpha_0+x_1\alpha_1+x_2\alpha_2\in \L$. Then $\Tr(x) = \lambda(x_0+x_1+x_2)$, where $\lambda=\Tr(\alpha_0)$, which is not zero as $\alpha_0,\alpha_1,\alpha_2$ are linearly independent over $\K$. Thus $x\in T_0$ if and only if $x_0+x_1+x_2=0$. 

Let $\beta =b_0\alpha_0+b_1\alpha_1+b_2\alpha_2\in \L^*$. Then, as in the proof of Lemma \ref{prop:Trquadform}, we see that $\beta\in T_0T_0$ if and only if there exists $x\in T_0$ such that $\Tr(\beta x^{\sigma+\sigma^2})=0$. Letting $x_0+x_1+x_2=0$, the left hand side of this equation is a quadratic form in $x_0,x_1$, with coefficients which are linear forms in $\K[b_0,b_1,b_2]$, as we now demonstrate.

Let
\begin{align*}
A&=\Tr(\alpha_0^3)=\alpha_0^3+\alpha_1^3+\alpha_2^3=\lambda^3 - 3\lambda\mu + 3\nu;\\
B&=\Tr(\alpha_0^{2+\sigma});\\
C&=\Tr(\alpha_0^{1+2\sigma}) = \mu-B;\\
D&=\Tr(\alpha_0^{1+\sigma+\sigma^2})= 3\nu.
\end{align*}
Then the quadratic is
\begin{align*}
\Tr(\beta x^{\sigma+\sigma^2})=&x_0^2(b_0(B+D-2C)+b_1(B+C-A-D)+b_2(C+D-2B))\\
&+x_1^2(b_0(B+C-A-D)+b_1(C+D-2B)+b_2(B+D-2C))\\
&+x_0x_1(b_0(4B-A-D-2C)+b_1(4C-2B-A-D)+b_2(A+3D-2B-2C)).
\end{align*}
Thus $\beta\in T_0T_0$ if and only if the discriminant of this quadratic form is a square in $\K$. The discriminant turns out to be
\[
\Delta = (\lambda^2 - 3\mu)^2((\lambda^2 - 2\mu)(b_0^2+b_1^2+b_2^2)-2(\lambda^2 - 4\mu)(b_0b_1+b_0b_2+b_1b_2)).
\]

\end{proof}

\begin{corollary}\label{cor:finite3} Consider  $\K=\Fq$, $\L=\F_{q^3}$, where $q$ is not a power of two or three. Then $T_0T_0\ne \F_{q^3}$.
\end{corollary}

\begin{proof}
From Proposition \ref{prop:cyclicgalois}, it suffices to show that the quadratic form $\Delta$ attains non-square values. Over a finite field, a quadratic form in three variables attains both square and non-square values, unless it is a scalar multiple of the square of a linear form.

Now $\Delta$ is the square of a linear form if and only if $\lambda^2=3\mu$. In this case, $f(x+\lambda/3) = x^3-(\nu-\lambda^3/27)=:x^3-d$. Let $\alpha$ be a root of $x^3-d$. Then the roots of $f$ are $\alpha-\lambda/3,\alpha\omega-\lambda/3,\alpha\omega^2-\lambda/3$, where $\omega\in \K$ is a root of $x^2+x+1$. But then the roots of $f$ do not form a $\K$-basis for $\L$, a contradiction. Thus $\Delta$ is not the square of a linear form, and thus $T_0T_0\ne \F_{q^3}$.
\end{proof}

{\bf Example} Corollary \ref{cor:finite3} gives an example of a degree $3$ extension of a finite field for which $T_0T_0$ is not the entire field. We will now see that there are degree $3$ extensions of infinite fields as well that have this property. For example, consider the field $\QQ(\alpha_0)$, where $\alpha_0=2\cos\frac{2\pi}{7}$ is a root of $x^3+x^2-2x-1$. This is a cyclic Galois extension of $\QQ$ of degree three. Let $\alpha_1,\alpha_2$ be the other two roots, which satisfy $\alpha_1= \alpha_0^2-2$, $\alpha_2=-\alpha_0^2-\alpha_0+1$.

Consider an element $x=x_0\alpha_0+x_1\alpha_1+x_2\alpha_2$. Then $\Tr(x)=-(x_0+x_1+x_2)$. Let $\beta =b_0\alpha_0+b_1\alpha_1+b_1\alpha_1$. Then taking $x_0+x_1+x_2=0$, some calculation shows that $\Tr(\beta x^{\sigma+\sigma^2})=0$ if and only if
\[    (2b_0 - b_2)x_0^2 + (3b_0 - 3b_2 + b_2)x_0x_1 +(- b_1 + 2b_2)x_1^2=0\]
This has a solution $(x_0,x_1)$ if and only if $2b_0=b_1$ or $$\Delta = 9(b_0^2+b_1^2+b_2^2)-10(b_0b_1+b_0b_2+b_1b_2)$$ is a square in $\QQ$. Taking for example $\beta=\alpha_0+\alpha_1+\alpha_2=-1$, we have $\Delta = -3$, a non-square, and hence $-1\notin T_0T_0$.

\subsection{$\K(\alpha)$, $\alpha$ a root of $x^3-d=0$. }
First note that if $f:X^3-d$ an irreducible polynomial over $\K$, and $\K$ is a field with $\mathrm{char}(\K)=3$ then $\L=\K(\alpha)$ is an inseperable extension, and the trace form is identically zero. Hence, in this case, we may restrict ourselves to $\cha(\K)\neq 3$.

\begin{lemma} \label{lemma3} Let $\K$ be an arbitrary field and suppose $f:X^3-d$ is an irreducible polynomial over $\K$. Let $\alpha$ be a root of $f$, and let $\L=\K(\alpha)$. Let $\omega$ be a primitive third root of unity, and define $\sigma\in \Aut(\K(\alpha,\omega)/\K)$ by $\sigma(\alpha)=\omega\alpha, \sigma(\omega)=\omega$. 

Then $\Tr_{\L/\K}(x)=x+x^\sigma+x^{\sigma^2}$, and $N_{\L/\K}(x) = x^{1+\sigma+\sigma^2}$.

Let $x=x_0+x_1\alpha+x_2\alpha^2$ be an arbitrary element in $\L$. Then $$\Tr_{\L/\K}(x)=3x_0.$$
\end{lemma}

\begin{proof}
Let $\E := \K(\alpha,\omega)$ be the splitting field of $f$. Then $\sigma$ generates a subgroup of $\Gal(\E/\K)$ of order $3$, with fixed field $\K(\omega)$, and $\E$ is a Galois extension of $\K(\omega)$. 

For any element $x$ of $\L$, its minimal polynomial over $\K$ is equal to its minimal polynomial over $\K(\omega)$. Hence we have that
\begin{align*}
\Tr_{\L/\K}(x) &= \Tr_{\E/\K(\omega)}(x) = x+x^\sigma+x^{\sigma^2},\\
N_{\L/\K}(x) &= N_{\E/\K(\omega)}(x) = x^{1+\sigma+\sigma^2}.
\end{align*}
Thus $\Tr_{\L/\K}(\alpha) = \alpha+\omega\alpha+\omega^2\alpha = (1+\omega+\omega^2)\alpha=0$, and similarly $\Tr_{\L/\K}(\alpha^2) =0$. Also $\Tr_{\L/\K}(1) = [\L:\K]=3.$

Let $x=x_0+x_1\alpha+x_2\alpha^2\in \L$. Then $$\Tr_{\L/\K}(x)= x_0\Tr_{\L/\K}(1) +x_1\Tr_{\L/\K}(\alpha) +x_2\Tr_{\L/\K}(\alpha^2) =3x_0.$$
\end{proof}

\begin{lemma} \label{cubicfin}
There exist an irreducible polynomial of the form $f(X)= X^3-d$ over $\F_q$ if and only if $q\equiv 1\mod 3$.
\end{lemma}
\begin{proof} An irreducible polynomial of the form $f$ exists if and only if there is an element $d$ in $\F_{q^3}$ that is not a cube. Consider the map  $g:X\mapsto X^3$. Every element of $\F_{q^3}$ is a cube if and only if $g$ is a permutation, so if and only if $x^3=1$ has only the trivial solution. This happens if and only if there are no elements of order $3$ in $\F_{q^3}^*$ which happens if and only if $3$ is not a divisor of $q^3-1$, this is, if and only if $q^3\not \equiv 1\mod 3$ if and only if $q\not\equiv 1\mod 3$.
This implies that not every element of $\F_{q^3}$ is a cube if and only if $q\equiv1\mod 3$.
\end{proof}

\subsubsection{$T_0T_0$}
\begin{proposition}\label{prop:cubicT0T0}
Let $\K$ be an arbitrary field with $\mathrm{char}(\K)\neq 3$ and let $f:X^3-d$ be an irreducible polynomial over $\K$. Let $\alpha$ be a root of $f$, let $\L=\K(\alpha)$. Then $\beta=b_0+b_1\alpha+b_2\alpha^2 \in T_0T_0$ if and only if 
$$b_1x_1^2 - b_0x_1x_2+db_2x_2^2=0$$
has a solution $(x_1,x_2)$ in $\K$.

\end{proposition}
\begin{proof}
Consider $\beta = b_0+b_1\alpha+b_2\alpha^2\in \L$. If $\beta=0$, then $\beta=0\cdot 0\in T_0T_0$. So assume that $\beta\neq 0$. We want to find solutions $x=x_0+x_1\alpha+x_2\alpha^2$, $y=y_0+y_1\alpha+y_2\alpha^2$ to $\Tr(x)=\Tr(y)=0, x\cdot y=\beta$. Since $\Tr(x)=3x_0$, we see that we require $x_0=0$ and $y_0=0$ in order for $\Tr(x)=\Tr(y)=0$ to hold.

Now since $x\ne 0$, $x\cdot y=\beta$ implies that $N(x)y = \beta x^{\sigma+\sigma^2}$, and so $\Tr(y)=0\Leftrightarrow \Tr(\beta x^{\sigma+\sigma^2})=0$. Hence there exists a solution if and only if there exist $x_1,x_2$ such that $\Tr(\beta (x_1\alpha+x_2\alpha^2)^{\sigma+\sigma^2})=0$.

Now 
\[
\beta x^{\sigma+\sigma^2} = (b_0+b_1\alpha+b_2\alpha^2)(\omega x_1\alpha+\omega^2 x_2\alpha^2)(\omega^2 x_1\alpha+\omega x_2\alpha^2),
\]
and multiplying this out we see that
\[
\Tr(\beta x^{\sigma+\sigma^2})/3d = (b_1x_1^2 - b_0x_1x_2)+d(b_2x_2^2).
\]
\end{proof}

\begin{corollary}\label{gevolg1}Let $\K$ be an arbitrary field with $\mathrm{char}(\K)\neq 2,3$ and let $f:X^3-d$ be an irreducible polynomial over $\K$. Let $\alpha$ be a root of $f$, let $\L=\K(\alpha)$. Then $\beta=b_0+b_1\alpha+b_2\alpha^2 \in T_0T_0$ if and only if 
\[b_0^2 -4db_1b_2
\]
is a square in $\K$.
\end{corollary}
\begin{proof}
Since the characteristic of $\K$ is not $2$, $b_1x_1^2 - b_0x_1x_2+db_2x_2^2=0$ has a solution if and only if
\[
b_0^2 -4db_1b_2
\]
is a square in $\K$.
\end{proof}

\begin{corollary} \label{n3T0T0finite2} Consider $\F_{q^3}$, $q=2^h$ $h$ even, and $\F_{q^3}=\F_q(\alpha)$ where $\alpha$ is a root of the irreducible polynomial $X^3-d$. Then $\beta=b_0+b_1\alpha+b_2\alpha^2 \in T_0T_0$ if and only if 
\[b_0=0\ \mathrm{or}\ \Tr(db_1b_2/b_0^2)=0.\]
\end{corollary}
\begin{proof} Recall that in $\F_{2^h}$, every element is a square. Hence, if $b_0=0$, the equation  $b_1x_1^2 - b_0x_1x_2+db_2x_2^2=0$ $(\ast)$ always has a solution. If $b_0\neq 0$, taking $X'=\frac{b_1}{b_0}X$, the equation $(\ast)$ is equivalent to $X'^2+X'+\frac{b_1b_2d}{b_0^2}$, which has a solution if and only if 
$\Tr(\frac{b_1b_2d}{b_0^2})=0$.
\end{proof}

\subsubsection{$T_0T_1$}
\begin{proposition}\label{prop:T0T1deg3}Let $\K$ be an arbitrary field  and let $f:X^3-d$ be an irreducible polynomial over $\K$. Then $T_0T_1=\L$.\end{proposition}
\begin{proof}
If $\cha(\K)=3$, the proposition is trivial, so we may assume that $\cha(\K)\neq 3$. Let $C_\beta(x)$ as in Lemma \ref{prop:Trquadform}. Let $\beta\in \L^*$. 

Suppose first that we can find $u$ with $C_\beta(u)=0$,  $a=\Tr_{\L/\K}(u) \ne 0$. Then we can let $x=a\beta/u, y=u/a$ so
$xy=\beta$ and $\Tr_{\L/\K}(x) = aC_\beta(u)/N_{\L/\K}(u)=0, \Tr_{\L/\K}(y) = \Tr_{\L/\K}(u)/a = 1$.

By Proposition \ref{prop:cubicT0T0}, we know that for an element $x=x_1\alpha+x_2\alpha^2$ with trace zero, $C_\beta(x) = 3d((b_1x_1^2 - b_0x_1x_2)+d(b_2x_2^2))$. Since $C_\beta$ is not identically zero, we can find a choice of $v\in \L^*$ with $C_\beta(v) \ne 0$ and $\Tr_{\L/\K}(v)=0$. Now put $x=\frac{bv}{N(v)}$. We see that $\Tr_{\L/\K}(x)=\frac{b}{N(v)}\Tr_{\L/\K}(v)=0$. Recall that $C_\beta(x)=\Tr_{\L/\K}(\beta/x)\cdot N(x)$ and let $b=\Tr_{\L/\K}(\beta/v)\cdot N(v)\neq 0$. We see that $\Tr_{\L/\K}(\frac{\beta}{x})=\Tr_{\L/\K}(\frac{\beta N(v)}{bv})=1$.

\end{proof}

\subsubsection{$T_1T_1$}
\begin{lemma}\label{n3T1T1} Let $\L$ and $\K$ be as in Proposition \ref{prop:cubicT0T0} and let $\beta=\beta_0+\beta_1\alpha+\beta_2\alpha^2$ be an arbitrary element in $\L$. Then $\beta\in T_1T_1$ if and only if 

\[
 b_0/3+3d(b_1x_1^2 - b_0x_1x_2-(b_1x_2+b_2x_1)/3)+3d^2(b_2x_2^2) = 1/27+dx_1^3+d^2x_2^3 - dx_1x_2 
\]
has a solution $(x_1,x_2)$.

\end{lemma}

\begin{proof}
Consider $\beta = b_0+b_1\alpha+b_2\alpha^2\in \L$. We want to find solutions to $\Tr(x)=1,\Tr(y)=1, x\cdot y=\beta$.

Take $x=1/3+x_1\alpha+x_2\alpha^2$. Then $x\cdot y=\beta\Leftrightarrow N(x)y = \beta x^{\sigma+\sigma^2}$, and so $\Tr(y)=1\Leftrightarrow \Tr(\beta x^{\sigma+\sigma^2})=N(x)$. Hence there exists a solution if and only if there exist $x_1,x_2$ such that

\[
 b_0/3+3d(b_1x_1^2 - b_0x_1x_2-(b_1x_2+b_2x_1)/3)+3d^2(b_2x_2^2)= 1/27+dx_1^3+d^2x_2^3 - dx_1x_2. 
\]\end{proof}

\begin{remark} 
We will see in Proposition \ref{T1T1one} that for odd characteristic, the equation of Lemma \ref{n3T1T1} does not always have a solution. However, the authors believe that there is always a solution to this equation when the characteristic is even. For $q=2^h$, $h\leq 6$ computer results confirm this belief. For $q=2^h$, $h\equiv 0 \mod 2$ and $h\not \equiv 0\mod 3$, it can be shown that the equation does not split into linear factors over some extension field, which implies that there are always solutions in $\F_{q^3}$. Since this proof is rather technical and does not lead to a full solution to the problem, we do not include it here.
\end{remark}

\subsection{$\K=\F_q$, $\L=\F_{q^3}$, $q=3^h$}

First note that, since the degree of the extension of $\F_{q^3}$ over $\F_q$ coincides with the characteristic $3$, it is always possible to find an irreducible polynomial of the form $$X^3-X+a.$$ 

\begin{lemma}\label{tracefin3}Let $\K=\F_q$, $q=3^h$ and let $f:X^3-X+a$ be an irreducible polynomial over $\K$. Let $\alpha$ be a root of $f$, let $\L=\K(\alpha)$. If $x=x_0+x_1\alpha+x_2\alpha^2$, then $$\Tr(x)=-x_2.$$

\end{lemma}
\begin{proof}If $\alpha$ is a root of this polynomial, then $\Tr(\alpha)=0$. By squaring both sides of the equation $\alpha^3-\alpha=-a$ we find that $(\alpha^2)^3-2(\alpha^2)^2+\alpha^2+a^2=0$, and hence, $\Tr(\alpha^2)=-1\neq 0$. Since $[\F_{q^3}:\F_{q}]=3$, $\Tr(1)=3=0$. The statement follows from the $\F_q$-linearity of the trace.

\end{proof}

\subsubsection{$T_0T_0$}

\begin{proposition} \label{n3T0T0finite3} Consider $\F_{q^3}$, where $q=3^h$ and $\F_{q^3}=\F_q(\alpha)$ with $\alpha$ is a root of the irreducible polynomial $X^3-X+a$. Then an element $\beta=b_0+b_1\alpha+b_2\alpha^2\in \F_{q^3}$ is contained in $T_0T_0$ if and only if  $$b_1^2 -b_0b_2$$ is a square in $\F_{q^3}$.
\end{proposition}
\begin{proof}By Lemma \ref{tracefin3}, if $\Tr(c_0+c_1\alpha+c_2\alpha^2)=0$, necessarily $c_2=0$. We conclude that every element in $T_0$ is of the form $c_0+c_1\alpha$.

Now consider an element $\beta=b_0+b_1\alpha+b_2\alpha^2$ in $\F_{q^3}$. Then $\beta\in T_0T_0$ if and only if
$$(x_0+x_1\alpha)(y_0+y_1\alpha)=b_0+b_1\alpha+b_2\alpha^2$$ has a solution. Since the minimal polynomial of $\alpha$ has degree $3$, it follows that $b_0+b_1Y+b_2Y^2$ has to split, which is the case if and only if $b_1^2-b_0b_2$ is a square. Conversely, if $b_1^2-b_0b_2$ is a square, $b_0+b_1Z+b_2Z^2$ splits and the solutions $Z_1=x_0+x_1\alpha$ and $Z_2=y_0+y_1\alpha$ are the elements in $T_0$ such that $Z_1Z_2=\beta$.

\end{proof}

\subsection{When is $1\in T_aT_b$?}
As for degree two extensions, we have seen in the previous section that $T_aT_b$ is not necessarily the field $\L$. We now wish to find out whether $1$ is in $T_aT_b$ or not. \subsubsection{$T_0T_0$}
\begin{corollary}\label{n3one} Let $[\L:\K]=3$.
If $\cha(\K)=3$, then $1\in T_0T_0$. If $\cha(\K)\neq 3$, then $1\in T_0T_0$ if and only if not every element of $\K$ is a cube.
\end{corollary}
\begin{proof} If $\alpha\in \K$, then $\Tr(\alpha)=3\alpha$. So if $\cha(\K)=3$, for any $\alpha\in \K^*$, $\Tr(\alpha)=\Tr(\alpha^{-1})=0$ and hence, $1\in T_0T_0$.
Now suppose $\cha(\K)\neq 3$ and suppose that $\Tr(\alpha)=\Tr(\alpha^{-1})=0$ for some $\alpha\in \K$. Since $\alpha\notin \F$, the minimal polynomial of $\alpha$ has to have degree $3$ and of the shape $X^3-d=0$ for some $d\in \K$. This polynomial is irreducible if and only if $d$ is not a cube in $\K$.
\end{proof}
\begin{corollary}\label{gevolg3} In $\F_{q^3}$, $1\in T_0T_0$ if and only if  $q\equiv 0,1\mod 3$.
\end{corollary}
\begin{proof} If $q=3^h$, then $1\in T_0$, so $1\in T_0T_0$. If $q\neq 3^h$ the statement follows directly from Lemma \ref{cubicfin} and Corollary \ref{n3one}.
\end{proof}

\subsubsection{$T_1T_1$}

\begin{proposition}\label{T1T1one} Let $[\L:\K]=3$. Then $1\in T_1T_1$ if and only if $\cha(\K)= 2$.
\end{proposition}
\begin{proof}
Suppose that there is an $\alpha \in \K$ such that $\Tr(\alpha)=1$ and $\Tr(\alpha^{-1})=1$. 

If $\alpha\in \K$, then $\Tr(\alpha)=3\alpha=1$ and $\Tr(\alpha^{-1})=3\alpha^{-1}=1$. This implies that $9=1$, hence, $\cha(\K)= 2$. If $\K$ has characteristic $2$, $\Tr(1)=\Tr(1^{-1})=1$, so $1\in T_1T_1$.

So if $\cha(\K)\neq 2$, $\K(\alpha)\neq \K$. Since $\K(\alpha)$ is a subfield of $\L$, which has degree $3$ over $\K$, $\K(\alpha)=\L$. 
Let $f$ be the minimal polynomial of $\alpha$, then $f$ is of the form
$X^3+aX^2+bX+c=0$. Since $\Tr(\alpha)=1$, $a=-1$. The element $\alpha^{-1}$ satisfies the equation $cX^3+bX^2+aX+1=0$, and hence, since $\Tr(\alpha^{-1})=1$, $b/c=-1$.
Hence, $f$ is of the form 
$$X^3-X^2+bX-b=0,$$
which is $(X-1)(X^2+b)$, and hence, never irreducible, a contradiction.

\end{proof}

\section{Extensions of degree $4$}\label{sec4}

\subsection{$\K(\alpha)$ where $\alpha$ is a root of $X^4+X^2+d$.}
First note that if $X^4+X^2+d$ is an irreducible polynomial over $\K$ and $\cha(\K)=2$, then $\L=\K(\alpha)$ is an inseperable extension. Hence, in this case, we may restrict ourselves to $\cha(\K)\neq 2$.

\begin{lemma}\label{trace4otherpol}
Let $\K$ be an arbitrary field with $\cha(\K)\neq 2$ and let $f:X^4+X^2+d$ be an irreducible polynomial over $\K$. Let $\alpha$ be a root of $f$, let $\L=\K(\alpha)$. If $x=x_0+x_1\alpha+x_2\alpha^2+x_3\alpha^3$, then $$\Tr(x)=4x_0-2x_2.$$
\end{lemma}

\begin{proof} We have that $\Tr(1)=4$. Since $\alpha$ is a root of $X^4+X^2+d$, $\Tr(\alpha)=0$, and $\Tr(1/\alpha)=0$. We see that $(\alpha^2)^2+\alpha^2+d=0$. Hence, $\Tr_{\K(\alpha^2)/ \K}(\alpha^2)=-1$. Moreover, $[\K(\alpha):\K(\alpha^2)]=2$, so $\Tr_{\K(\alpha)/ \K}(\alpha^2)=\Tr_{\K(\alpha^2)/ \K(\alpha)}\Tr_{\K(\alpha)/ \K}(\alpha^2)=2\cdot(-1)$. We have that $\alpha^3=-\alpha-d/\alpha$, and hence, $\Tr(\alpha^3)=0$. The statement follows by the $\K$-linearity of $\Tr$.
\end{proof}

\subsubsection{$T_0T_0$}
\begin{proposition}\label{T0T04} Let $\L$ and $\K$ be as in Lemma \ref{trace4otherpol}. Then $\L=T_0T_0$.
\end{proposition}

\begin{proof} It is clear that $0\in T_0T_0$, so let $\beta=b_0+b_1\alpha+b_2\alpha^2+b_3\alpha^3\neq 0$.
First suppose that $b_1=b_3=0$. Let $x=-b_0/d\alpha-b_1/d\alpha^3$, then $\Tr(x)=0$. We see that $\beta/x=(\alpha+\alpha^3)$, and $\Tr(\beta/x)=0$.

 Let $x=(2db_3-b_1)\alpha +(b_3-2b_1)\alpha^3=x_1\alpha+x_3\alpha^3$. Suppose that  $2db_3-b_1=0=b_3-2b_1$. Then either $b_1=b_3=0$ or $b_1b_3\neq 0$, but then $d=1/4$ but $X^4+X^2+1/4$ is reducible, a contradiction. 
Suppose that $b_1=b_3=0$, so $\beta=b_0+b_2\alpha$. Then it is readily checked that $\beta=x\cdot y$ with $x=(-b_0/d \alpha-b_1/d\alpha^3)$ and $y=(\alpha+\alpha^3)$. It is clear that $\Tr(x)=\Tr(y)=0$.

So assume $x\neq 0$. Since $x_0=x_2=0$, we have that $\Tr(x)=0$. We find that $\frac{\alpha}{x}=\frac{(x_3-x_1)+x_1\alpha^2}{x_1x_3-x_1^2-x_3^2d}$. Using this, a straightforward but tedious calculation shows that $\Tr(\beta/x)=0$.
\end{proof}

\begin{corollary}\label{n4T0T0fin2}
Let $\L=\F_{q^4}$, with $q\equiv 1,3\mod 4$. Then $T_0T_0= \L$.
\end{corollary}

\begin{proof} Since $q$ is odd, we can take an element $d$ in $\F_q$ such that $1-4d$ is a non-square. This implies that the polynomial $X^4+X^2+d$ is irreducible. Take this as minimal polynomial for $\alpha$ and $\F_{q^4}=\F_q(\alpha)$. Then $T_0T_0=\F_{q^4}$ by Proposition \ref{T0T04}.
\end{proof}

\subsubsection{$T_0T_1$}
\begin{proposition}\label{T0T14a}
Let $\L$ and $\K$ be as in Lemma \ref{trace4otherpol}. Then $\L=T_0T_1$.
\end{proposition}
\begin{proof} Let $\beta=b_0+b_1\alpha+b_2\alpha^2+b_3\alpha^3$. Recall from Lemma \ref{trace4otherpol} that $\Tr(\beta)=4b_0-2b_2$.

If $b_1=b_3=0$, then $\beta\in \K(\alpha^2)$, which is a quadratic extension of $\K$. From Proposition \ref{n2T0T1}, noting that $\Tr_{\K(\alpha)/\K}(x) = 2\Tr_{\K(\alpha^2)/\K}$ for any $x\in \K(\alpha^2)$, then $\K(\alpha^2)/\K\subset T_0T_1$. Thus we can assume $b_1=b_2=b_3=0$, that is, $\beta\in \K$. Indeed it suffices to show that $1\in T_0T_1$, which can be directly verified by taking $x=\frac{16}{4d+3}(1-2\alpha^2)$.

If $b_3=0$, and $b_1\neq 0$, then let $x=4b_1\alpha$ implying $\Tr(x)=0$. Let $y=\frac{1}{4}+\frac{b_2d-b_0}{4b_1d}\alpha-\frac{b_0}{4b_1d}\alpha^3$, then $\Tr(y)=1$ and $x\cdot y=\beta$.
Now let $b_3\neq 0$. Suppose that $b_1b_3-b_1^2-db_3^2=0$. Then $d=-\frac{b_1^2}{b_3^2}+\frac{b_1}{b_3}$. Put $y_0=\frac{b_1}{b_3}$. It follows that $X^4+X^2+d=X^4+X^2-y_0^2+y_0$, and hence, that $X^4+X^2+d=(X^2+y_0)(X^2-y_0+1)$, a contradiction since $X^4+X^2+d$ is irreducible.

It follows that $b_1b_3-b_1^2-db_3^2\neq 0$. Let $x=4b_1\alpha+4b_3\alpha^3$, then $\Tr(x)=0$. Let $y=\frac{1}{4}+\frac{b_0b_3(d-1)+b_0b_1-b_1b_2d+b_2b_3d}{(b_1b_3-b_1^2-db_3^2)4d}\alpha+\frac{b_0b_1-b_0b_3+b_2b_3d}{(b_1b_3-b_1^2-db_3^2)4d}\alpha^3$, then $\Tr(y)=1$. It is not hard to check that $x\cdot y=\beta$.
\end{proof}

\begin{corollary}\label{n4T0T1fin2}
Let $\L=\F_{q^4}$, $q\equiv 1,3\mod{4}$. Then $T_0T_1= \L$.
\end{corollary}

\subsection{$\K(\alpha)$ where $\alpha$ is a root of $X^4-d$.}
First note that if $f:X^4-d$ is an irreducible polynomial over $\K$ and $\cha(\K)=2$, then $\L=\K(\alpha)$ is an inseperable extension. Hence, also in this case, we may restrict ourselves to $\cha(\K)\neq 2$.

The proof of the following Lemma goes along the same lines as that of Lemma \ref{lemma3}.

\begin{lemma} \label{lemma4} Let $\K$ be an arbitrary field and let $f:X^4-d$ be an irreducible polynomial over $\K$. Let $\alpha$ be a root of $f$, let $\L=\K(\alpha)$ and let $i$ be a primitive fourth root of unity. Let $\E=\K(\alpha,\omega)$ be the splitting field of $f$, and let $\sigma\in\Gal(\E/\K)$ be defined by $\sigma(\alpha)=i\alpha$, $\sigma(i)=i$. 

Then 
\begin{align*}
\Tr_{\L/\K}(x)&= x+x^\sigma+x^{\sigma^2}+x^{\sigma^3},\\
N_{\L/\K}(x)&= x^{1+\sigma+\sigma^2+\sigma^3}.
\end{align*}

Let $x=x_0+x_1\alpha+x_2\alpha^2$ be an arbitrary element in $\L$. Then 
\[
\Tr_{\L/\K}(x)=4x_0\]
\end{lemma}

\subsubsection{$T_0T_0$}

\begin{proposition}\label{prop:quarticT0T0}
Let $\K$ be an arbitrary field with $\mathrm{char}(\K)\neq 2$ and let $f:X^4-d$ be an irreducible polynomial over $\K$. Let $\alpha$ be a root of $f$, let $\L=\K(\alpha)$. Then $\beta=b_0+b_1\alpha+b_2\alpha^2+b_3 \alpha^3 \in T_0T_0$ if and only if the cubic curve
\begin{align*}
C_\beta(x_1,x_2,x_3) = &(b_0x_1^2x_2 - b_1x_1^3)\\
&+d(b_0x_2x_3^2+b_1(x_1x_3^2-x_2^2x_3)+b_2(x_2^3-2x_1x_2x_3)+b_3(x_1^2x_3-x_1x_2^2))\\
&+d^2(-b_3x_3^3)
\end{align*}
has a $\K$-rational point.
\end{proposition}

\begin{proof}
Consider $\beta = b_0+b_1\alpha+b_2\alpha^2+b_3\alpha^3\in \L$. We want to find solutions to $\Tr(x)=\Tr(y)=0, x\cdot y=\beta$.

Take $x=x_1\alpha+x_2\alpha^2+x_3\alpha^3$. Then $x\cdot y=\beta\Rightarrow N(x)y = \beta x^{\sigma+\sigma^2+\sigma^3}$, and so $\Tr(y)=0\Leftrightarrow \Tr(\beta x^{\sigma+\sigma^2+\sigma^3})=0$. Hence there exists a solution if and only if there exist $x_1,x_2,x_3$ such that $\Tr(\beta x^{\sigma+\sigma^2+\sigma^3})=0$.

Now 
\[
\beta x^{\sigma+\sigma^2+\sigma^3} = (b_0+b_1\alpha+b_2\alpha^2+b_3\alpha^3)(i x_1\alpha-x_2\alpha^2-ix_3\alpha^3)(- x_1\alpha+x_2\alpha^2-x_3\alpha^3)(-i x_1\alpha-x_2\alpha^2+ix_3\alpha^3),
\]
and multiplying this out we see that
\begin{align*}
\Tr(\beta x^{\sigma+\sigma^2+\sigma^3})/4d &= (b_0x_1^2x_2 - b_1x_1^3)\\
&+d(b_0x_2x_3^2+b_1(x_1x_3^2-x_2^2x_3)+b_2(x_2^3-2x_1x_2x_3)+b_3(x_1^2x_3-x_1x_2^2))\\
&+d^2(-b_3x_3^3).
\end{align*}
We have a solution if and only if the cubic curve defined by the right hand side of this equation has a $\K$-rational point, proving the claim.
\end{proof}

\begin{proposition}\label{n4T0T0} Let $\L$ and $\K$ be as in Proposition \ref{prop:quarticT0T0}.
Then $\L=T_0T_0$.
\end{proposition}

\begin{proof}
By Proposition \ref{prop:quarticT0T0}, it suffices to show that the cubic form $C_\beta$ has a nontrivial $\K$-rational point.  

$x_1 = db_3/b_1, x_2 = 0, x_3 =1$ is such a point if $b_1\ne 0$, while $x_1 = 1, x_2 = 0, x_3 = 0$ is such a point if $b_1=0$.
\end{proof}

For a finite field of odd order $q$, an irreducible of the form $X^4-d$ exists if and only if $q\equiv 1\mod 4$. Hence we have the following immediate corollary.

\subsubsection{$T_0T_1$}
\begin{proposition}\label{T0T14}
Let $\L$ and $\K$ be as in Proposition \ref{prop:quarticT0T0}.
Then $\L=T_0T_1$.
\end{proposition}
\begin{proof} Let $\beta=b_0+b_1\alpha+b_2\alpha^2+b_3\alpha^3$. First suppose that $b_2=0$. If $b_1=0$, then let $x=\frac{1}{4}+\frac{b_0}{4b_3d}\alpha^3$, then $\Tr(x)=1$. Let $y=4b_3\alpha^3$, then $\Tr(y)=0$. We see that $\beta=x\cdot y$. If $b_1\neq 0$, then let $x=\frac{1}{4}+\frac{b_3}{4b_1}\alpha^2+\frac{b_0}{4b_1d}\alpha^3$. We have that $\Tr(x)=1$. Let $y=4b_1\alpha$, then $\Tr(y)=0$. We see that $\beta=x\cdot y$.

Now suppose that $b_2\neq 0$. Since $d$ is a non-square, $b_0^2\neq b_2^2d$. Let $x=\frac{1}{4}+\frac{b_0}{4b_2d}\alpha^2$, then $\Tr(x)=1$. Let $y=4(\frac{b_0b_2b_3d-b_1b_2^2d}{b_0^2-b_2d}\alpha+b_2\alpha^2+\frac{b_0b_1b_2-b_2^2b_3d}{b_0^2-b_2d}\alpha^3)$, then $\Tr(y)=0$. It is not hard to check that $x\cdot y=\beta$.
\end{proof}

\subsection{When is $1\in T_0T_0$?}
\begin{proposition} \label{n4fieldeven} Let $[\L:\K]=4$. If $\cha(\K)=2$, then $1\in T_0T_0$.
\end{proposition}
\begin{proof}
Pick $\alpha\in \K$, then $\Tr(\alpha)=4\alpha=0=\Tr(\alpha^{-1})$.
\end{proof}

\begin{corollary} \label{n4finone}Let $\L=\F_{q^4}$ and $\K=\Fq$. Then $1\in T_0T_0$.
\end{corollary}
\begin{proof} For $q$ odd this follows from Corollary \ref{n4T0T0fin2} and for $q$ even from Proposition \ref{n4fieldeven}.\end{proof}

\begin{proposition}\label{n4subfield} Let $[\L:\K]=4$ with $\cha(\K)\neq 2$. Then $1\in T_0T_0$ if and only if $\L$ has a quadratic subfield.
\end{proposition}
\begin{proof}
Suppose that there is an $\alpha\in \L$ such that $\Tr(\alpha)=\Tr(\alpha^{-1})=0$.
\begin{itemize}
\item If $\K(\alpha)=\K$, then $\Tr(\alpha)=4\alpha$. Using that $\cha(\K)\neq 2$, this gives us $\alpha=0$, a contradiction.
\item If $\L=\K(\alpha)$, then the minimal polynomial of $\alpha$ is of the form $X^4+aX^2+b$ for some $a,b\in \F$. From \cite[Section 14.6 pg. 615]{DF},
the Galois group of $X^4+aX^2+b$ is either $D_4,V_4$ or $C_4$. Each of these groups $G$ has subgroups $H,H'$ with $H\leq H'\leq G$ and $[G:H']=2$, $[H':H]=2$. Hence, $H'$ gives rise to an intermediate field $\E$ with $\K< \E< \L$, and $\deg[\E:\K]=2$.
\item If $\K(\alpha)\neq \K$ and $\K(\alpha)\neq \L$, then $\K(\alpha)$ is a quadratic subfield of $\L$.
\end{itemize}
Vice versa, suppose that $\L$, with $\mathrm{char}(\K)\neq 2$ has a quadratic subfield $\E$, i.e. $[\E:\K]=2$. Let $\gamma\in \E\setminus \K$ be an element with minimal polynomial $X^2+a_1X+a_2$. Put $\delta=\gamma+\frac{a_1}{2}$, then $\delta\in \E\setminus \K$ and 
$$\delta^2=\gamma^2+a_1\gamma+\frac{a_1^2}{4},$$
and hence, $\delta^2=\frac{a_1^2}{4}+a_2$.

Now, $\Tr_{\E/\K}(\delta)=0$ since $\delta$ satisfies $X^2-\frac{a_1^2}{4}-a_2=0$ which is irreducible over $\F$. This implies that $\Tr(\delta)=\Tr_{\L/ \E}\Tr_{\E/\K}(\delta)=\Tr_{\L/\K}(0)=0$.
Moreover, $\delta^{-1}$ satisfies the equation $(\frac{a_1^2}{4}-a_2)X^2-1=0$, which is irreducible over $\F$, and hence, $\Tr_{\E/\K}(\delta^{-1})=0$. It follows that $\Tr(\delta)=0$.
\end{proof}

\begin{corollary} Let $\L$ be a field with $\cha(\K)\neq 2$ such that $\L=\K(\alpha)$,  $[\L:\K]=4$,  with $f(\alpha)=0$ the minimal polynomial of $\alpha$ and $gal(f)=S_4$, then $1\notin T_0T_0$.
\end{corollary}
\begin{proof} A field extension of degree $4$ with Galois group $S_4$ cannot have an intermediate field: it is impossible to find subgroups $H\leq H'\leq S_4$ such that $[S_4:H']=2$ and $[H':H]=2$.
\end{proof}

\section{Extensions of degree at least $5$ for finite fields}\label{sec5}

\begin{theorem}\label{Th1}
Let $\alpha$ be an element of $\F_{q^n}^*$, $n\geq 5$, 
(or $n=4$ and $q=2,3$), and let $a,b\in \F_{q}$. Then $\alpha$ can be written as $\alpha=x\cdot y$ where $x,y \in \F_{q^n}, \Tr(x)=a, \Tr(y)=b$.
\end{theorem}

\begin{proof} Let $\beta_a$ be an element of $\F_{q^n}$ with $\Tr(\beta_a)=a$ and let $\beta_b$ be an element of $\F_{q^n}$ with $\Tr(\beta_b)=b$. By
the additive form of Hilbert's theorem 90,

we know that writing $\alpha=x\cdot y$,  where $\Tr(x)=a$ and $\Tr(y)=b$ is equivalent to finding elements $z,t \in \F_{q^n}$ such that 
$$\alpha=(z^q-z+\beta_a)(t^q-t+\beta_b).$$
We rewrite this as
$$z^q-z=\frac{\alpha}{t^q-t+\beta_b}-\beta_a.$$
The right hand side is an element $u(t)$ in $\F_{q^n}(t)$ 
which has $q$ simple poles in $\overline{\F}_{q}$, the algebraic closure of
$\F_{q}$.
Hence, $u$ cannot be written as $f(t)^p-f(t), f(t) \in \overline{\F}_{q}(t)$ 
as $u$ has simple poles and any poles of $f(t)^p-f(t)$ have order 
divisible by $p$. So 
the curve $$C:z^q-z=u(t)$$ is an irreducible cover of the projective $t$-line.
Applying the Hurwitz formula  
(see  e.g. \cite[Proposition III.7.10]{Stichtenoth} for a more general
calculation) the genus $g$ of $C$ is

$$g=\frac{q-1}{2}(-2+2q)=(q-1)^2.$$

The Hasse-Weil bound \cite[Theorem V.2.3]{Stichtenoth} 
$$|q^n+1-N|\leq 2gq^{n/2},$$
then gives us a bound on the number of points $N$ on $C(\F_{q^n}$) :

\begin{align}q^n+1-2(q-1)^2q^{n/2}\leq N.\end{align}

The points where $z$ or $t$ have poles have to be discarded and these are
at most $2q$ points.

We see that, if $n\geq 5$,
or $n=4$ and $q=2,3$, $N> 2q$, and hence, there are elements $z,t \in \F_{q^n}$ such that
$$z^q-z=\frac{\alpha}{t^q-t+\beta_b}-\beta_a.$$

\end{proof}

\section{Applications}\label{sec6}
\label{applications}

We now show various applications of the ideas developed in the first part of the paper. Most are based on the following observation: the curve

\[
\frac{\Tr(y)}{y}=f(x)
\]
has an $\Fn$-rational point with $yf(x) \ne 0$ if and only if the curve
\[
y^q-y=\frac{1}{f(x)}-e=g(x)
\]
has an $\Fn$-rational point with $f(x)\ne 0$, where $e \in \Fn$ is some fixed element satisfying $\Tr(e)=1$. 
This follows firstly by noticing that, if $\Tr(y)\ne 0$, $z=y/\Tr(y)$ satisfies $\Tr(z)=1$ and conversely any 
$z \in \Fn, \Tr(z)=1$ is of the form $z=y/\Tr(y)$ (e.g. with $y=z$) and, secondly, using the fact that the elements of trace
zero are precisely those of the form $y^q-y$ for some $y$.

More generally, we show various problems which correspond to determining the existence of $\Fn$-rational points on curves of the form
\[
\frac{L_1(y)}{y} = g(x),,
\]
in particular when $g(x) = \frac{L_2(x)}{x} +a$, for linearised polynomials $L_1(x),L_2(x)$.

\subsection{Disjoint linear sets}
Let $\K$ be a field, and $\L$ an extension of $\K$ of degree $n$. A {\it $\K$-linear set} is a set of points of the projective space $\PG(k-1,\L)=\P(\L^k)$ defined by a $\K$-subspace of $\L^k$. Such sets have been studied in particular over finite fields, due for example to their uses in constructing (multiple) {\it blocking sets} \cite{blo2,blo}, {\em KM-arcs} \cite{KM}, and, as we see later, finite semifields. We refer to \cite{FQ11},\cite{olga} for further details. 

\begin{definition}
Suppose $\K$ is a subfield of $\L$. An {\it $\F_{q}$-linear set of rank $s$} in $\PG(k-1,\L)$ is a set
\[
L(U) := \{\la u\ra_{\L} : u\in U^{\ast}\}
\]
for some $\K$-subspace $U$ of $\L^k$ with $\dim_{\K}(U)=s$. 
\end{definition}

Here $\la u\ra_{\L}$ denotes the projective point in $\PG(k-1,\L)$, corresponding to the vector $u$, where the notation reflects the fact that all ${\L}$-multiples of $u$ define the same projective point. When $u=(u_0,\ldots,u_{k-1})$ is a vector in $\L^k$, then $\la u\ra_{\L}=\la (u_0,\ldots,u_{k-1})\ra_{\L}$ will be denoted as $(u_0,\ldots,u_{k-1})_{\L}$, or $(u_0,\ldots,u_{k-1})_{q^n}$ if $\L=\Fn$ is finite.

\begin{definition}
Let $L(U)$ be a $\K$-linear set in $\PG(1,\L)$. The {\it weight} of a point $\in L(U)$ is defined as the dimension of the intersection of $U$ and $\la u\ra_{\L}$ when both are viewed as $\K$-vector spaces.

A {\it club} is a $\K$-linear set of rank $s$ in $\PG(1,\L)$ containing a point of weight $s-1$, which is called the {\em head} of the club. A $\K$-linear set is said to be {\it scattered} if all of its points have weight at most one.
\end{definition}

For the remainder of this section we will focus on the case $k=2$, $s=n$, $\L=\Fn$, $\K=\Fq$. An important question in this setting is the existence problem of {\it disjoint} linear sets of rank $n$. We will show how the results of the previous sections give immediate results in this regard.

For a linearised polynomial $f(x)$, we define
\begin{align*}
\Gamma_f &:= \{(x,f(x))_{q^n}:x\in \Fn^*\}\\
\overline{\Gamma}_f &:= \{(f(x),x)_{q^n}:x\in \Fn^*\}
\end{align*}

Every linear set of rank $n$ in $\PG(1,q^n)$ is $\PGL$-equivalent to one of the form $\Gamma_f$, for some linearised polynomial $f(x)$. Two such linear sets $\Gamma_f,\Gamma_g$ intersect if and only if 
\[
xg(y)-yf(x)=0
\]
for some nonzero $x,y\in \Fn$, or equivalently, the curve
\[
\frac{g(y)}{y}=\frac{f(x)}{x}
\]
has nonzero $\Fn$-rational points.

Furthermore, $\Gamma_f$ and $\overline{\Gamma}_g$ intersect if and only if the curve
\[
\frac{g(y)}{y}\frac{f(x)}{x}=1
\]
has nonzero $\Fn$-rational points.

\subsubsection{Disjoint clubs}

\begin{lemma} \label{lemma3b} Every club of rank $n$ in $\PG(1,q^n)$ with head $(1,0)_{q^n}$, and disjoint from $(0,1)_{q^n}$ is of the form $$\{(z,\gamma\Tr(z))_{q^n}|z\in \F_{q^n}^*\}$$ for some $\gamma\in \Fn^*$.
\end{lemma}
\begin{proof} Every $\F_{q}$-linear set $S$, disjoint from $(0,1)_{q^n}$, can be written as $(x,f(x))_{q^n}$, where $f$ is a linear map. If $S$ is a club with head $(1,0)_{q^n}$, then $f(x)$ has $q^{n-1}$ roots, and hence, $f(x)$ is a rank $1$ map. Every rank $1$ map on $\F_{q^n}$ can be represented as $x\mapsto \alpha\Tr(\beta x)$ for some $\alpha,\beta\in \F_{q^n}^*$.
Now put $\beta x =z$, then $$\{(x,\alpha\Tr(\beta x))_{q^n}|x\in \F_{q^n}^*\}=\{(\frac{z}{\beta},\alpha\Tr(z))_{q^n}|z\in \F_{q^n}^*\}.$$
This in turn equals $$\{(z,\alpha\beta\Tr(z))_{q^n}|z\in \F_{q^n}^*\}=\{(z,\gamma\Tr(z))_{q^n}|z\in \F_{q^n}^*\},$$
 for $\gamma=\alpha\beta$. \end{proof}

\begin{lemma} Let $C_1$ and $C_2$ be two disjoint clubs of rank $n$ in $\PG(1,q^n)$. Then they are, up to $\PGL$-equivalence, equal to 
$$C_1=\{(x,\Tr(x))_{q^n}|x\in \F_{q^n}^*\}$$ and $$C_2=\{(\alpha\Tr(y),y)_{q^n}|y\in \F_{q^n}^*\}$$ for some $\alpha\in \Fn^*$.

\end{lemma}
\begin{proof} As $\PGammaL(2,q^n)$ acts $3$-transitively on $\PG(1,q^n)$, we may pick the head of $C_1$ to be $(1,0)_{q^n}$ and the head of $C_2$ to be $(0,1)_{q^n}$. As $C_1$ and $C_2$ are disjoint, by Lemma \ref{lemma3b}, they are of the form
$\{(x,\gamma\Tr(x))_{q^n}|x\in \F_{q^n}^*\}$ and $\{\gamma'\Tr(y),y)_{q^n}|y\in \F_{q^n}^*\}$.

Now elements of the stabiliser of $(1,0)_{q^n}$ and $(0,1)_{q^n}$ in $\PGL(2,q^n)$ are induced by matrices of the form $\begin{pmatrix}a&0\\0&d\end{pmatrix}$, which map $\{(x,\gamma\Tr(x))_{q^n}|x\in \F_{q^n}^*\}$ to $\{(ax,d\gamma\Tr(x))_{q^n}|x\in \F_{q^n}^*\}$. So the element $\begin{pmatrix}1&0\\0&1/\gamma\end{pmatrix}$ induces a map from $\{(x,\gamma\Tr(x))_{q^n}|x\in \F_{q^n}^*\}$ onto $C_1=\{(x,\Tr(x))_{q^n}|x\in \F_{q^n}^*\}$ and $\{\gamma'\Tr(y),y)_{q^n}|y\in \F_{q^n}^*\}$ onto
$\{\gamma'\Tr(y),y/\gamma)_{q^n}|y\in \F_{q^n}^*\}=\{\gamma\gamma'\Tr(y),y)_{q^n}|y\in \F_{q^n}^*\}$. Putting $\alpha=\gamma\gamma'$ yields that the latter equals $C_2$.
\end{proof}

\begin{proposition} The sets $C_1=\{(x,\Tr(x))_{q^n}|x\in \F_{q^n}^*\}$ and $C_2=\{(\alpha\Tr(y),y)_{q^n}|z\in \F_{q^n}^*\}$, $\alpha\in \Fn^*$, are disjoint if and only if $$\frac{x}{\Tr(x)}
\frac{y}{\Tr(y)}=\alpha$$ does not have a solution, which happens if and only if $\alpha\notin T_1T_1$. \end{proposition}
\begin{proof} $C_1$ and $C_2$ are disjoint if and only if 
$$(x,\Tr(x))_{q^n}\neq(\alpha\Tr(y),y)_{q^n}$$ for all $x,y$.
This happens if and only if for all $x,y\in \F_{q^n}^*$,
$$\alpha\Tr(y)/y\neq x/\Tr(x),$$or
$$\alpha\neq \frac{x}{\Tr(x)}\frac{y}{\Tr(y)}.$$
Since $\Tr(\frac{x}{\Tr(x)})=\frac{\Tr(x)}{\Tr(x)}=1$, we find that this happens if and only if $\alpha\notin T_1T_1$.

\end{proof}
\begin{corollary} \label{nodisjointclubs}  There exist disjoint clubs in $\PG(1,q^2)$ and in $\PG(1,q^3)$ when $q$ is odd. There are no disjoint clubs in $\PG(2^4), \PG(3^4)$ and $\PG(1,q^n)$, $n\geq 5$.
\end{corollary}

\subsubsection{Linear sets disjoint from a club}
In Corollary \ref{nodisjointclubs}, we have seen that when $n\geq 5$, there cannot be a club disjoint from a club in $\PG(1,q^n)$. In this Subsection, we will show that, if a linear set is defined by a linearised poynomial of relatively small degree, it can never be disjoint from a club.

\begin{theorem}\label{thm:disj} Let $$C_1=\{(x,\Tr(x))_{q^n}|x\in \F_{q^n}^*\}$$ and $$C_2=\{(f(y),y)_{q^n}|y\in \F_{q^n}^*\},$$where $f$ is an $\F_q$-linearised polynomial of $q$-degree $d$.
If $d<n/2-1$, then $C_1$ and $C_2$ have at least one point in common.
\end{theorem}

\begin{proof} We have seen that a linear set and a club, up to $\PGL$-equivalence, or equal to 
$$C_1=\{(x,\Tr(x))_{q^n}|x\in \F_{q^n}^*\}\ \mathrm{and}\ C_2=\{(f(y),y)_{q^n}|y\in \F_{q^n}^*\},$$ where $f$ is an $\F_q$-linearised polynomial.
We will show that the equation
\[\frac{x}{\Tr(x)}=\frac{f(y)}{y}\] has at least one solution if $d<n/2-1$.

Note that the elements $u$ of $\F_{q^n}$ of the form $x/\Tr(x)$ are
exactly those with $\Tr(u)=1$.
As in the proof of Theorem \ref{Th1}, the elements of
$u$ of $\F_{q^n}$ with $\Tr(u)=1$ are those of the form
$z^q-z+\beta_1$ for some $z\in \F_{q^n}$, where $\beta_1$ is a fixed element with $\Tr(\beta_1)=1$.
The equation then becomes
\begin{align}
z^q-z=\frac{f(y)}{y}-\beta_1
\label{eq}
\end{align}

Since $f$ has $q$-degree $d$, its degree is $q^d$, and hence, the degree of the polynomial $\frac{f(y)}{y}-\beta_1$ is $q^d-1$, which is coprime to the 
characteristic of $\F_{q^n}$.

Again, as in Theorem \ref{Th1}, the Hurwitz formula shows that this
equation defines a curve over $\F_{q^n}$ of genus $(q^d-2)(q-1)/2$
and the Weil bound gives the following estimate for the number $N$ of
$\F_{q^n}$-points of the curve.

$$|N-q^n|\leq (q^d-2)(q-1)q^{n/2}.$$

The unique point at infinity does not give a solution to the equation
(\ref{eq}), so we need $N > 1$ to guarantee a solution and
this is true under the assumption that $d < n/2 -1$.

\end{proof}

{\bf Example.}
The standard example for a scattered linear set  in $\PG(1,q^n)$ is given by $S=\{ (x,x^q)_{q^n}\mid x\in \Fn^*\}$. It follows from Theorem \ref{thm:disj} there are no clubs in $\PG(1,q^n)$, disjoint from $S$.

\subsection{Finite Semifields}

A {\it finite presemifield} is a finite-dimensional division algebra over a finite field, where multiplication is not assumed to be associative. If a multiplicative identity exists, it is called a {\it finite semifield}. We will omit the word finite from now on.

In \cite{BEL}, a geometric construction for semifields was introduced. In \cite{BELrank}, this was developed further, and the notion of the {
\it BEL-rank} of a semifield was introduced. Of relevance to this paper are semifields of BEL-rank two.

\begin{definition}
A semifield $\S$ is said to have {\it BEL-rank two} if there exist linearised polynomials $L_1(x),L_2(x)$ over $\Fn$ such that
\[
x\circ y = L_1(x)L_2(y)- xy
\]
defines a presemifield isotopic to $\S$.
\end{definition}

We denote the algebra with this multiplication by $\S_{L_1,L_2}$. Given a linearised polynomial $L_1(x)$, we consider whether there exists a linearised polynomial $L_2(x)$ such that 
\[
L_1(x)L_2(y)- xy \ne 0
\]
for all $x,y\in \Fn^*$. Equivalently, we need to find for which $L_2$ the curve
\[
\frac{L_2(y)}{y} = \frac{x}{L_1(x)} 
\]
does not have $\Fn$-rational points. 

There are some known constructions for semifields of this form, as well as some sporadic examples found by computer presented in \cite{BELrank}. One of the sporadic examples has $L_1(x)=\Tr_{\FF_{2^6}:\FF_{2^2}}(x)$, and so Theorem \ref{thm:disj} gives a partial answer towards resolving whether or not this example lives in an infinite family.

\begin{corollary}
If  $\S_{\Tr,L_2}$ defines a semifield, then $\deg(L_2)\geq q^{n/2 - 1}$.
\end{corollary}

\subsection{PN functions}

In this section we illustrate an application of our results to PN functions.

\begin{definition}
A function $F:\Fn\mapsto\Fn$ is said to be {\it planar}, or {\it perfect nonlinear (PN)}, if for all  $y\in \Fn^*$, the map
\[
F_y(x) := F(x+y)-F(x)-F(y)
\]
is a permutation.
\end{definition}
We refer to \cite{KyOz} for background on PN functions. Most known PN functions have a special form, known as Dembowski-Ostrom polynomials. We review the main results here.

\begin{definition}
A {\it Dembowski-Ostrom} polynomial is a polynomial in $\Fn[x]$ of the form
\[
\sum_{i,j=0}^{n-1}a_{ij}x^{q^i+q^j}.
\]
\end{definition}

\begin{result}
Suppose $L_1(x),L_2(x)$ are linearised polynomials. Then $L_1(x)L_2(x)$ is Dembowski-Ostrom, and is EA-equivalent to $xL(x)$ for some linearised $L(x)$.
\end{result}

\begin{result}
Suppose $L(x)$ is a linearised polynomial. Then the following are equivalent:
\begin{itemize}
\item[(1)]
The function $x(L(x)+\alpha x)$ is PN;
\item[(2)]
The multiplication $x\circ y := xL(y)+yL(x)+2\alpha xy$ on $\Fn$ defines a commutative semifield;
\item[(3)]
The curve $\frac{L(y)}{y}=-\frac{L(x)}{x}+2\alpha$ has no $\Fn$-rational points.
\end{itemize}
\end{result}

It is well known that PN functions of Dembowski-Ostrom type are closely related to commutative semifields in odd characteristic. 
Namely, if $(\Fn,\circ)$ is a commutative semifield, then the function $F(x):= x\circ x$ is PN and Dembowski-Ostrom.

In \cite{KyOz}, \cite{YangZhuFeng}, the following was shown, using techniques similar to ours.
\begin{result} Let $n\geq 4$ and $a\in \F_{q^n}$. Then the function $x(\Tr_n(x)+a x)$ is not planar on $\F_{q^n}$.
\end{result}
Precise conditions on when $x(\Tr_n(x)+a x)$ is planar were established in \cite{CoulHend}.

We are now ready to derive a similar result for the function $\Tr_n(x)^2+a x^2$.
\begin{theorem} 
Let $n\geq 5$ and $a\in \F_{q^n}$. Then the function $F(x)= \Tr_n(x)^2+a x^2$ is not planar on $\F_{q^n}$.
\end{theorem}

\begin{proof}
$F(x)$ is planar if and only if $\Tr(x)\Tr(y)+axy\ne 0$ for all $x,y\in \F_{q^n}$ with $xy\ne 0$. This is clearly never the case when $a$ is zero. If $a$ is nonzero, this occurs if and only if $-a^{-1}\notin T_1T_1$, since a solution would imply that $a^{-1} = \left(\frac{x}{\Tr(x)}\right)\left(\frac{x}{\Tr(x)}\right)\in T_1T_1$. But by Theorem \ref{Th1}, $T_1T_1=\F_{q^n}$ for all $n\geq 5$, proving the result.
\end{proof}

\begin{remark}
Note that if $-a=b^2$ is a square in $\Fn$, then $\Tr_n(x)^2+a x^2=(\Tr_n(x)-bx)(\Tr_n(x)+bx)$, which is EA-equivalent to $xL(x)$ for some $L(x)$ if and only if $\Tr(b)\ne 0$. However when $-a$ is not a square in $\Fn$, or when $\Tr(\sqrt{a})\ne 0$, we have a new nonexistence result on PN functions of this type.
\end{remark}

\subsection{Irreducible polynomials with prescribed coefficients.}\label{prescribed}
The Hansen-Muller conjecture states that, exept for 5 small exceptional cases, for every $n$ and $q$ there exists a monic irreducible polynomial of degree $n$ over $\F_q$ with one prescribed coefficient. It is assumed that the constant term (if preassigned) is nonzero. This conjecture was settled apart from a finite number of cases in \cite{Wan}, and finally completely solved in \cite{ham}.

Many related results have followed after this, e.g. restricting to primitive polynomials, or specifying more prescribed coefficients. Most results are asymptotic: for example, one can prove that, if $q$ is suffficiently large, and the number of prescribed coefficients is not too large compared to $n$, there exists a monic irreducible polynomial with these prescribed coefficients \cite{Ha}. In particular, it is shown that 

\begin{result}\cite[Theorem 1.3]{Ha} Suppose $n\geq 8$, $q\geq 16$ and $r\leq n/4-log_q(n)-1$. Then there exists a monic irreducible polynomial of degree $n$ with $r$ prescribed coefficients, except when $0$ is assigned in the constant term. We conclude the same when $q\geq 5$, $n\geq 97$, and $r\leq n/5$; or when $n\geq 52$,$r\leq n/10$ for arbitrary $q$.
\end{result}

A similar result in the case where the second coefficient and one other coefficient are prescribed are found in \cite{panario}. These are valid for $n\geq 22$ and $q> 97$, and for $q\leq 97$ and $n>22$ specified in function of $q$, given in a table \cite{panario}.

Using the results developed in this paper, we can add some missing small cases to the previous theorems by showing the following:

\begin{theorem} Let $a,b\in \F_q$, $q=p^h$. Let $n\neq 2,3$ be a prime number. If $n^2\not\equiv ab \mod p$, then there exists a monic irreducible polynomial of degree $n$ of the form

$$X^n+c_1 X^{n-1}+c_2X^{n-2}+\cdots+c_{n-1}X+c_n$$
with $c_1=a$ and $c_{n-1}/c_n=b$.
\end{theorem}
\begin{proof} By Theorem \ref{Th1}, $T_{-a}T_{-b}=\F_{q^n}$. Hence, $1\in T_{-a}T_{-b}$ implying that there is an element $\alpha\in \F_{q^n}$ with $\Tr(\alpha)=-a$ and $\Tr(\frac{1}{\alpha})=-b$. Since $n$ is prime, the minimal polynomial $f(X)$ of $\alpha$ has degree $n$ or degree $1$. If $\alpha\notin \F_q$, the minimal polynomial of $\alpha$ has degree $n$ and $c_1=-\Tr(\alpha)=a$. We see that $\frac{1}{\alpha}$ satisfies the equation $c_nX'^n+c_{n-1}X'^{n-1}+\ldots+c_1X'+1$ and it follows that $c_{n-1}/c_n=-\Tr(\frac{1}{\alpha})=b$.

Now, suppose that the element $\alpha$ such that $\Tr(\alpha)=a$ and $\Tr(\frac{1}{\alpha})=b$ is in $\F_q$. Then $\Tr(\alpha)=n\alpha=a$ and $\Tr(\frac{1}{\alpha})=n/\alpha=b$. It follows that, if $n\not\equiv 0 \mod p$, $\alpha=a/n=n/b$. And hence, $ab\equiv n^2 \mod p$. If $n\equiv0\mod p$, then $\Tr(\alpha)=0=a$ and $\Tr(\frac{1}{\alpha})=0=b$, and hence, $n^2\equiv ab\mod p$. 
\end{proof}

{\bf Acknowledgment: }This research was initiated when the first author visited the School of Mathematics and Statistics at the University of Canterbury, New Zealand. His stay was funded by a Strategic Reseach grant of the College of Engineering of the University of Canterbury. The second and third author were supported by the Marsden Fund Council administered by the Royal Society of New Zealand.

John Sheekey: 

School of Mathematics and Statistics\\
University College Dublin \\
Belfield, Dublin 4\\
Ireland\\
john.sheekey@ucd.ie
\bigskip

Jos\'e Felipe Voloch:

School of Mathematics and Statistics\\
University of Canterbury\\
Private Bag 4800\\
8140 Christchurch\\
New Zealand\\
felipe.voloch@canterbury.ac.nz

\bigskip

Geertrui Van de Voorde:

School of Mathematics and Statistics\\
University of Canterbury\\
Private Bag 4800\\
8140 Christchurch\\
New Zealand\\
geertrui.vandevoorde@canterbury.ac.nz


\begin{thebibliography}{99}
\bibitem{blo2} S. Ball, A. Blokhuis, M. Lavrauw. Linear $(q+1)$-fold blocking sets in $\PG(2,q^4)$. {\em Finite Fields Appl.} {\bf 6 (4)} (2000), 294--301.

\bibitem{BEL} S. Ball, G. Ebert and M. Lavrauw. A geometric construction of finite semifields. {\em J. Algebra} {\bf 311 (1)} (2007),117--129.
\bibitem{CoulHend}
R. Coulter and M. Henderson. On a conjecture on planar polynomials of the form $x(\Tr(x)-ux)$. {\em Finite Fields Appl.} {\bf 21} (2013), 30--34. 
\bibitem{KM} M. De Boeck and G. Van de Voorde. A linear set view on KM-arcs. {\em J. Algebraic Combin.} {\bf 44 (1)} (2016), 131--164.

\bibitem{DF}
D. Dummit and R. Foote. {\em Abstract algebra.} Third edition. Wiley, 2004. 

\bibitem{Ha} J. Ha. Irreducible polynomials with several prescribed coefficients.{\em Finite Fields Appl.} {\bf 40} (2016), 10--25.
\bibitem{ham}K. Ham and G. Mullen. Distribution of irreducible polynomials of small degrees over finite fields. {\em Math. Comp.} {\bf 67 (221)} (1998),  337--341.


\bibitem{KyOz} G. Kyureghyan and F. Ozbudak. Planar products of linearized polynomials. WCC 2011 - Workshop
on coding and cryptography, Apr 2011, Paris, France. pp.351-360, 2011.

%

%
%
%
%
%
%
%
%
%
%
%
%
%

\bibitem{BELrank}
M. Lavrauw and J. Sheekey. The BEL-rank of finite semifields, {\em Des. Codes Cryptogr.} {\bf 84} (3), 345--358
\bibitem{FQ11} M. Lavrauw and G. Van de Voorde. Field reduction in finite geometry. {\em Topics in finite fields}. Contemp. Math., 632, Amer. Math. Soc., Providence, RI, 2015.

\bibitem{blo} G. Lunardon. Linear $k$-blocking sets. {\em Combinatorica} {\bf 21 (4)} (2001), 571--581.
%

\bibitem{panario} D. Panario and G. Tzanakis. A generalization of the Hansen-Mullen conjecture on irreducible polynomials over finite fields. {\em Finite Fields Appl.} {\bf 18 (2)} (2012), 303--315. 
\bibitem{olga} O. Polverino. Linear sets in finite projective spaces. {\em Discrete Math.} {\bf 310 (22)} (2010), 3096--3107.
\bibitem{Stichtenoth} H. Stichtenoth. {\em Algebraic Function Fields and Codes.} Universitext, Springer, Berlin, 1993.


\bibitem{YangZhuFeng}
M. Yang, S. Zhua and K. Feng. Planarity of mappings $x(\Tr(x)-\frac{\alpha}{2} x)$ on finite fields. {\em Finite Fields Appl.} {\bf 23} (2013), 1--7. 



\bibitem{Wan} D. Wan. Generators and irreducible polynomials over finite fields.{\em Math. Comput.} {\bf 66 (219)} (1997),  1195--1212.
\end{thebibliography}
\end{document}